\documentclass[12pt]{amsart}
\usepackage{amsfonts,amssymb,amscd,amsmath,enumerate,verbatim,color}
\usepackage[latin1]{inputenc}
\usepackage{amscd}
\usepackage{latexsym}

\usepackage[dvipdfmx]{graphicx}
\usepackage{mathptmx}

\def\NZQ{\mathbb}               

\def\ZZ{{\NZQ Z}}
\def\RR{{\NZQ R}}

\def\frk{\mathfrak}               

\def\Phi{{\frk N}}
\def\ab{{\mathbf a}}

\def\eb{{\mathbf e}}
\def\tb{{\mathbf t}}
\def\vb{{\mathbf v}}
\def\ub{{\mathbf u}}
\def\wb{{\mathbf w}}
\def\xb{{\mathbf x}}
\def\yb{{\mathbf y}}

\def\opn#1#2{\def#1{\operatorname{#2}}} 
\opn\gr{gr}

\def\Ec{{\mathcal E}}

\def\Jc{{\mathcal J}}
\def\Gc{{\mathcal G}}
\def\Fc{{\mathcal F}}

\def\Pc{{\mathcal P}}

\def\Cc{{\mathcal C}}

\def\Vol{{\textnormal{Vol}}}

\newtheorem{Theorem}{Theorem}[section]
\newtheorem{Lemma}[Theorem]{Lemma}
\newtheorem{Corollary}[Theorem]{Corollary}
\newtheorem{Proposition}[Theorem]{Proposition}

\theoremstyle{definition}
\newtheorem{Remark}[Theorem]{Remark}

\newtheorem{Problem}[Theorem]{Problem}
\newtheorem{Conjecture}[Theorem]{Conjecture}

\let\epsilon\varepsilon
\let\phi=\varphi
\let\kappa=\varkappa
\opn\dis{dis}
\opn\height{height}
\opn\dist{dist}
\def\pnt{{\raise0.5mm\hbox{\large\bf.}}}

\opn\Lex{Lex}
\opn\conv{conv}

\begin{document}

\title{Enriched chain polytopes}
\author{Hidefumi Ohsugi and Akiyoshi Tsuchiya}
\address{Hidefumi Ohsugi,
	Department of Mathematical Sciences,
	School of Science and Technology,
	Kwansei Gakuin University,
	Sanda, Hyogo 669-1337, Japan} 
\email{ohsugi@kwansei.ac.jp}

\address{Akiyoshi Tsuchiya,
Graduate school of Mathematical Sciences,
University of Tokyo,
Komaba, Meguro-ku, Tokyo 153-8914, Japan} 
\email{akiyoshi@ms.u-tokyo.ac.jp}


\begin{abstract}
Stanley introduced a lattice polytope $\mathcal{C}_P$ arising from a finite poset $P$, which is called the chain polytope of $P$.
The geometric structure of $\mathcal{C}_P$ has good relations with the combinatorial structure of $P$. In particular, the Ehrhart polynomial of $\mathcal{C}_P$ is given by the order polynomial of $P$.
In the present paper, associated to $P$, we introduce a lattice polytope $\mathcal{E}_{P}$, which is called the enriched chain polytope of $P$, and investigate geometric and combinatorial properties of this polytope.
By virtue of the algebraic technique on Gr\"{o}bner bases, we see that 
$\mathcal{E}_P$ is a reflexive polytope with a flag regular unimodular triangulation.
Moreover, the $h^*$-polynomial of $\mathcal{E}_P$ is equal to the $h$-polynomial of a flag triangulation of a sphere.
On the other hand, by showing that the Ehrhart polynomial of $\mathcal{E}_P$ coincides with   the left enriched order polynomial of $P$, 
it follows from  works of Stembridge and Petersen that the $h^*$-polynomial of $\mathcal{E}_P$ is $\gamma$-positive.
Stronger, we prove that the $\gamma$-polynomial of $\mathcal{E}_P$ is equal to the $f$-polynomial of a flag simplicial complex.
\end{abstract}

\maketitle

\section*{Introduction}

A {\em lattice polytope} $\Pc  \subset \RR^n$ of dimension $n$ is a convex polytope all of whose vertices have integer coordinates.
Given a positive integer $m$, we define
$$L_{\Pc}(m)=|m \Pc \cap \ZZ^n|.$$
The study on $L_{\Pc}(m)$ originated in Ehrhart \cite{Ehrhart} who proved that $L_{\Pc}(m)$ is a polynomial in $m$ of degree $n$ with the constant term $1$.
We say that $L_{\Pc}(m)$ is the \textit{Ehrhart polynomial} of $\Pc$.
The generating function of the lattice point enumerator, i.e., the formal power series
$$\text{Ehr}_\Pc(x)=1+\sum\limits_{k=1}^{\infty}L_{\Pc}(k)x^k$$
is called the \textit{Ehrhart series} of $\Pc$.
It is well known that it can be expressed as a rational function of the form
$$\text{Ehr}_\Pc(x)=\frac{h^*(\Pc,x)}{(1-x)^{n+1}}.$$
The polynomial $h^*(\Pc,x)$ is a polynomial in $x$ of degree at most $n$ with nonnegative integer coefficients (\cite{Stanleynonnegative}) and it
is called
the \textit{$h^*$-polynomial} (or the \textit{$\delta$-polynomial}) of $\Pc$. 
Moreover, one has $\Vol(\Pc)=h^*(\Pc,1)$, where $\Vol(\Pc)$ is the normalized volume of $\Pc$.

In \cite{twoposetpolytopes}, Stanley introduced a class of  lattice polytopes associated with finite partially ordered sets. 
Let $P = [n]:=\{1,2,\dots,n\}$ be a partially ordered set (poset, for short).
An {\em antichain} of $P$ is a subset of $P$ consisting of pairwise incomparable elements of $P$.
Note that the empty set $\emptyset$ is an antichain of $P$.
The {\em chain polytope} $\Cc_P$ of $P$
is the convex hull of 
$$
\{
\eb_{i_1} + \dots + \eb_{i_k}
:
\{i_1, \dots,  i_k\} \mbox{ is an antichain of } P
\},
$$
where $\eb_i$ is $i$-th unit coordinate vector of $\RR^n$ and the empty set $\emptyset$ corresponds to the origin ${\bf 0}$ of $\RR^n$.
Then $\Cc_P$ is a lattice polytope of dimension $n$.
There is a close interplay between the combinatorial structure of $P$ and the geometric structure of $\Cc_P$.
For instance, it is known the Ehrhart polynomial $L_{\Cc_P}(m)$ and the order polynomial $\Omega_P(m)$ are related by $\Omega_P(m+1)=L_{\Cc_P}(m)$.
On the other hand, $\Cc_P$ has many interesting properties.
In particular,  the toric ring of $\Cc_P$ is an algebra with straightening laws, and thus the toric ideal possesses a squarefree quadratic initial ideal (\cite{HibiLi}).
Moreover,  $\Cc_P$ is of interest in representation theory (\cite{represent})  and statistics (\cite{Sul}).

Now, we introduce a new class of lattice polytopes associated with posets.
 The {\em enriched chain polytope} $\Ec_P$  is the convex hull of 
$$
E(P)=
\{
\pm \eb_{i_1} \pm \dots \pm \eb_{i_k}
:
\{i_1, \dots,  i_k\} \mbox{ is an antichain of } P
\}.
$$
Then $\dim \Ec_{P} =n$.
It is easy to see that $\Ec_P$ is centrally symmetric
(i.e., for any facet $\Fc$ of $\Ec_P$, $-\Fc$ is also a facet of $\Ec_P$), and the origin  ${\bf 0}$ of $\RR^n$ is the unique interior lattice point of $\Ec_P$.
In the present paper, we investigate geometric and combinatorial properties of $\Ec_P$.

A lattice polytope $\Pc \subset \RR^n$ of dimension $n$ is called \textit{reflexive} if the origin of $\RR^n$ is a unique lattice point belonging to the interior of $\Pc$ and its dual polytope
\[\Pc^\vee:=\{\yb \in \RR^n  :  \langle \xb,\yb \rangle \leq 1 \ \text{for all}\  \xb \in \Pc \}\]
is also a lattice polytope, where $\langle \xb,\yb \rangle$ is the usual inner product of $\RR^n$.
It is known that reflexive polytopes correspond to Gorenstein toric Fano varieties, and they are related to
mirror symmetry (see, e.g., \cite{mirror,Cox}).
In each dimension there exist only finitely many reflexive polytopes 
up to unimodular equivalence (\cite{Lag})
and all of them are known up to dimension $4$ (\cite{Kre}).
Recently, several classes of reflexive polytopes are constructed by the virtue of the algebraic technique on Gr\"{o}bner bases (c.f., \cite{HTomega,HTperfect,harmony}).
By showing the toric ideal of $\Ec_P$ possesses a squarefree quadratic initial ideal (Theorem \ref{thm:grobner}), in Section \ref{sec:reflexive}, we prove the following.
\begin{Theorem}
	\label{thm:flag}
	Let $P=[n]$ be a poset.
	Then $\Ec_P$ is a reflexive polytope
	having a flag regular unimodular triangulation
such that each maximal simplex 
contains the origin as a vertex.
\end{Theorem}

We now turn to the discussion of the Ehrhart polynomial and the $h^*$-polynomial of $\Ec_P$.
In fact, the Ehrhart polynomial of $\Ec_P$ is equal to a combinatorial polynomial associated to $P$.
In Section \ref{sec:enriched}, we prove the following.
\begin{Theorem}
	\label{ehrhartpolynomial}
	Let $P = [n]$ be a naturally labeled poset. 
	Then one has
		\[L_{\Ec_P}(m)
	=
	\Omega_{P}^{(\ell)}(m),
	\]
	where 	$\Omega_{P}^{(\ell)}(m)$ is the left enriched order polynomial of $P$.
\end{Theorem}
In \cite{StembridgeEnriched}, Stembridge developed the theory of {\em enriched $(P,\omega)$-partitions}, in analogy with Stanley's theory of $(P,\omega)$-partitions.
In the theory, {\em enriched order polynomials} were introduced.
On the other hand, Petersen \cite{Petersen} introduced slightly different notion, {\em left enriched $(P,\omega)$-partitions} and {\em left enriched order polynomials}. Please refer to Section \ref{sec:enriched} for the details.
Therefore,  from Theorem \ref{ehrhartpolynomial}  we call $\Ec_{P}$ the ``enriched" chain polytope of $P$.

Next, we discuss Gal's Conjecture for enriched chain polytopes.
Gal \cite{Gal} conjectured that the $h$-polynomial  of a flag triangulation of a sphere is $\gamma$-positive. On the other hand, Theorem~\ref{thm:flag} implies that $h^*(\Ec_P,x)$ coincides with the $h$-polynomial of a flag triangulation of a sphere (Corollary \ref{cor:flag}).
Therefore, $h^*(\Ec_P,x)$ is expected to be $\gamma$-positive.
From works of Stembridge \cite{StembridgeEnriched} and Petersen \cite{Petersen} and Theorem \ref{ehrhartpolynomial}, we can obtain the following.
\begin{Theorem}
	\label{maintheorem}
	Let $P = [n]$ be a naturally labeled poset. 
	Then the $h^*$-polynomial of $\Ec_P$ is
	$$
	h^*(\Ec_P, x)
	=  (x+1)^n \ W_{P}^{(\ell)}  \left(  \frac{4x}{(x+1)^2} \right),
	$$
	where $W_{P}^{(\ell)}(x)$ is  the left peak polynomial of $P$.
	In particular, $h^*(\Ec_P, x)$ is $\gamma$-positive.
	Moreover, $	h^*(\Ec_P, x)$ is real-rooted if and only if $W_{P}^{(\ell)}  (x)$ is real-rooted.
\end{Theorem}

\noindent
Note that $W_{P}^{(\ell)}  (x)$ is not necessarily real-rooted (\cite{Stembridge}).

Finally, we discuss Nevo-Petersen's Conjecture for enriched chain polytopes.
In \cite{NevoPetersen}, Nevo and Petersen made a stronger conjecture than Gal's Conjecture.
They conjectured that the $h$-polynomial  of a flag triangulation of a sphere is equal to the $f$-polynomial of a simplicial complex. 
In other words, the coefficients of the $\gamma$-polynomial of a flag triangulation of a sphere satisfy Kruskal-Katona inequalities.
In Section \ref{sec:complex}, we construct explicit flag simplicial complexes whose $f$-polynomials are the $\gamma$-polynomials of enriched chain polytopes (Theorem \ref{thm:fpoly}).


\subsection*{Acknowledgment}
The authors were partially supported by JSPS KAKENHI 18H01134 and 16J01549.

\section{squarefree quadratic Gr\"{o}bner bases}
\label{sec:reflexive}
In this section, we prove Theorem \ref{thm:flag}.
First, we see the geometric structure of $\Ec_{P}$ of a finite poset $P=[n]$.
Given $\varepsilon = (\varepsilon_1,\ldots, \varepsilon_n) \in \{-1,1\}^n$,
let ${\mathcal O}_\varepsilon$ denote the closed orthant 
$\{ (x_1,\ldots, x_n) \in \RR^n  : x_i  \varepsilon_i \ge 0 \mbox{ for all } i \in [n]\}$.
Let ${\mathcal L}(P)$ denote the set of linear extensions of $P$.
It is known \cite[Corollary 4.2]{twoposetpolytopes} that 
the normalized volume of the chain polytope $\Cc_P$ is $|{\mathcal L}(P)|$.

\begin{Lemma}
\label{keylemma}
Work with the same notation as above.
Then each $\Ec_P \cap {\mathcal O}_\varepsilon$ is the convex hull of the set
$E(P) \cap {\mathcal O}_\varepsilon$
and unimodularly equivalent to the chain polytope $\Cc_P$ of $P$.
In particular, the normalized volume of $\Ec_P$ is 
$\Vol(\Ec_P)=2^n \Vol(\Cc_P) = 2^n | {\mathcal L}(P) |$.
\end{Lemma}

\begin{proof}
It is enough to show that $\Ec_P \cap {\mathcal O}_\varepsilon \subset {\rm Conv} (E(P) \cap {\mathcal O}_\varepsilon)$.
Let $\xb = (x_1,\ldots,x_n) \in \Ec_P \cap {\mathcal O}_\varepsilon$.
Then $\xb = \sum_{i=1}^s \lambda_i \ab_i$, where $\lambda_i >0$,
$\sum_{i=1}^s \lambda_i = 1$, and each $\ab_i$ belongs to $E(P)$.
Suppose that 
$k$-th component of $\ab_i$ is positive and 
$k$-th component of $\ab_j$ is negative.
Then we replace $\lambda (\ab_i  + \ab_j)$ in 
$\xb = \sum_{i=1}^s \lambda_i \ab_i$ with $\lambda((\ab_i -\eb_k) + (\ab_j + \eb_k))$, where $\lambda = \min \{\lambda_i, \lambda_j  \}$
and $\ab_i -\eb_k ,\ab_j + \eb_k \in E(P)$.
Repeating this procedure finitely many times,
we may assume that
$k$-th component of each vector $\ab_i$ is nonnegative (resp.~nonpositive)
if $x_k \ge 0$ (resp.~$x_k \le 0$).
Then each $\ab_i$ belongs to $E(P) \cap {\mathcal O}_\varepsilon$ and hence
$\xb \in {\rm Conv} (E(P) \cap {\mathcal O}_\varepsilon)$.
\end{proof}

In order to show that $\Ec_P$ is reflexive and has a flag regular unimodular triangulation,
we use an algebraic technique on Gr\"obner bases.
We recall basic materials and notation on toric ideals.
Let $K[{\bf t^{\pm1}}, s] = K[t_{1}^{\pm1}, \ldots, t_{n}^{\pm1}, s]$
be the Laurent polynomial ring in $n+1$ variables over a field $K$. 
If $\ab = (a_{1}, \ldots, a_{n}) \in \ZZ^{n}$, then
${\bf t}^{\ab}s$ is the Laurent monomial
$t_{1}^{a_{1}} \cdots t_{n}^{a_{n}}s \in K[{\bf t^{\pm1}}, s]$. 
Let  $\Pc \subset \RR^{n}$ be a lattice polytope and $\Pc \cap \ZZ^n=\{\ab_1,\ldots,\ab_d\}$.
Then, the \textit{toric ring}
of $\Pc$ is the subalgebra $K[\Pc]$ of $K[{\bf t^{\pm1}}, s] $
generated by
$\{\tb^{\ab_1}s ,\ldots,\tb^{\ab_d}s \}$ over $K$.
We regard $K[\Pc]$ as a homogeneous algebra by setting each $\text{deg } \tb^{\ab_i}s=1$.
Let $K[\xb]=K[x_1,\ldots, x_d]$ denote the polynomial ring in $d$ variables over $K$.
The \textit{toric ideal} $I_{\Pc}$ of $\Pc$ is the kernel of the surjective homomorphism $\pi : K[\xb] \rightarrow K[\Pc]$
defined by $\pi(x_i)=\tb^{\ab_i}s$ for $1 \leq i \leq d$.
It is known that $I_{\Pc}$ is generated by homogeneous binomials.
See, e.g., \cite{sturmfels1996}.
The following lemma follows from the same argument in \cite[Proof of Lemma 1.1]{HMOS}.

\begin{Lemma}
\label{genten_yowaku}
	Let $\Pc \subset \RR^n$ be a lattice polytope of dimension $n$ such that the origin of $\RR^n$ is contained 
	in its interior.
	Suppose that any lattice point in $\ZZ^n$ is a linear integer combination of the lattice points in $\Pc$.
If 
there exists a monomial order such that
the initial ideal is generated by squarefree monomials which do not contain
the variable corresponding to the origin,
then $\Pc$ is reflexive and has a regular unimodular triangulation.
\end{Lemma}

Let $P=[n]$ be a poset and let $R_P$ denote the polynomial ring 
$$
R_P =K[x_{i_1,\dots,i_k}^\varepsilon :  
\{i_1, \dots, i_k\} \mbox{ is an antichain of } P, \ \varepsilon \in \{-1, 1\}^k]
$$
in $|\Ec_P \cap \ZZ^n|$ variables over a field $K$.
In particular, the origin corresponds to the variable $x_\emptyset$.
Then the toric ideal $I_{\Ec_P}$ of $\Ec_P$ is the kernel of a ring homomorphism 
 $\pi: R_P \rightarrow K[\tb^{\pm1}, s]$
defined by
$\pi(x_{i_1,\dots,i_k}^{(\varepsilon_1,\dots, \varepsilon_k)}) = t_{i_1}^{\varepsilon_1}
\dots t_{i_k}^{\varepsilon_k} s$.
In addition, 
$$
I_{\Ec_P} \cap K[x_{i_1,\dots,i_k}^{(1,\dots,1)} :  
\{i_1, \dots, i_k\} \mbox{ is an antichain of } P]$$
is the toric ideal $I_{\Cc_P}$ of the chain polytope $\Cc_P$ of $P$.
Hibi and Li essentially constructed a squarefree quadratic  initial ideal of $I_{\Cc_P}$ in \cite{HibiLi}.
Let $\Jc(P)$ be the finite distributive lattice consisting
of all poset ideals of $P$, ordered by inclusion.
Given a subset $Z \subset P$, let $\max (Z)$ denote the set of all maximal elements of $Z$.
Then $\max (Z)$ is an antichain of $P$.
For a subset $Y$ of $P$, the poset ideal of $P$ generated by $Y$ is
the smallest poset ideal of $P$ which contains $Y$.
Given poset ideals $I, J \in \Jc(P)$, 
let $I * J$ denote the poset ideal of $P$
generated by $\max(I \cap J) \cap (\max(I) \cup \max(J))$.
Then Hibi and Li proved in \cite[Proof of Theorem~2.1]{HibiLi} that
 the set of all binomials of the form
$$
x_{\max(I)}^{(1,\dots,1)} x_{\max(J)}^{(1,\dots,1)} 
-
x_{\max(I \cup J)}^{(1,\dots,1)} x_{\max(I*J)}^{(1,\dots,1)} 
$$
is a Gr\"obner basis of $I_{\Cc_P}$ with respect to a monomial order $<$.
Here the initial monomial of each binomial is the first monomial.
It is known \cite[Proposition~1.11]{sturmfels1996} that there exists a nonnegative weight vector
$\wb \in \RR^{|\Jc(P)|}$ such that ${\rm in}_\wb (I_{\Cc_P}) = {\rm in}_<(I_{\Cc_P}) $.
Then we define the weight vector $\wb_P$ on $R_P$ such that
the weight of each variable $x_{i_1,\dots,i_k}^\varepsilon$
with respect to $\wb_P$ is the weight of the variable $x_{i_1,\dots,i_k}^{(1,\dots,1)}$ with respect to $\wb$.
In addition, let $\wb_{\rm card}$ be
the weight vector on $R_P$ such that
the weight of each variable $x_{i_1,\dots,i_k}^\varepsilon$
with respect to $\wb_{\rm card}$ is $k$.
Fix any monomial order $\prec$ on $R_P$ as a tie-breaker.
Let $<_P$ be a monomial order on $R_P$ such that $u <_P v$
if and only if one of the following holds:
\begin{itemize}
\item
The weight of $u$ is less than that of $v$ with respect to $\wb_{\rm card}$;
\item
The weight of $u$ is the same as that of $v$ with respect to $\wb_{\rm card}$,
and the weight of $u$ is less than that of $v$ with respect to $\wb_P$;
\item
The weight of $u$ is the same as that of $v$ with respect to $\wb_{\rm card}$
and $\wb_P$, and $u \prec v$.
\end{itemize}

\begin{Theorem}
	\label{thm:grobner}
Work with the same notation as above.
Let $\Gc$ be the set of all binomials
\begin{equation}
\label{saisyo}
x_{i_1,\dots,i_k}^{(\varepsilon_1,\dots, \varepsilon_k)} x_{j_1,\dots, j_\ell}^{(\mu_1,\dots, \mu_\ell)}
-
x_{i_1,\dots, i_{p-1}, i_{p+1}, \dots,i_k}^{(\varepsilon_1,\dots, \varepsilon_{p-1}, \varepsilon_{p+1}, \dots,\varepsilon_k)} x_{j_1,\dots,j_{q-1}, j_{q+1},\dots, j_\ell}^{(\mu_1,\dots, \mu_{q-1}, \mu_{q+1}, \dots, \mu_\ell)},
\end{equation}
where $i_p = j_q$ and $\varepsilon_p \neq \mu_q$,
\begin{equation}
\label{niban}
x_{i_1,\dots,i_k}^{(\varepsilon_1,\dots, \varepsilon_k)} x_{i_{k+1},\dots,i_{k+\ell}}^{(\varepsilon_{k+1},\dots, \varepsilon_{k+\ell})}
-
x_{j_1,\dots,j_{k'}}^{(\mu_1,\dots, \mu_{k'})} 
x_{j_{k'+1},\dots,j_{k'+\ell'}}^{(\mu_{k'+1},\dots, \mu_{k'+\ell'})}
\ (\neq 0),
\end{equation}
where
\begin{itemize}
\item[{\rm (a)}]
For any $p,q$ such that $i_p = i_q$,
we have $\varepsilon_p =\varepsilon_q${\rm ;}
\item[{\rm (b)}]
For any $p,q$ such that $i_p = j_q$,
we have $\varepsilon_p =\mu_q${\rm ;}
\item[{\rm (c)}]
For some $I, J \in \Jc(P)$, we have
$\max (I)= \{i_1,\dots,i_k\}$, 
$\max (J)= \{i_{k+1},\dots,i_{k+\ell}\}$,
$\max (I \cup J) =  \{j_1,\dots,j_{k'}\}$,
$\max (I*J) =  \{j_{k'+1},\dots,j_{k'+\ell'}\}$.
\end{itemize}
Then $\Gc$ is a Gr\"obner basis of $I_{\Ec_P}$ with respect to a monomial order $<_P$.
The initial monomial of each binomial is the first monomial.
In particular, the initial ideal is generated by squarefree quadratic monomials
which do not contain the variable $x_\emptyset$.
\end{Theorem}

\begin{proof}
It is easy to see that any binomial of type (\ref{saisyo}) belongs to $I_{\Ec_P}$.
In analogy to \cite{HibiLi}, any binomial of type (\ref{niban}) belongs to $I_{\Ec_P}$.
Hence $\Gc$ is a subset of $I_{\Ec_P}$.
For a binomial $u-v$ of type (\ref{niban}), let $\pi(u) = t_1^{a_1} \dots t_n^{a_n} s^2$ and $\pi(v) = t_1^{b_1} \dots t_n^{b_n} s^2$.
Since $u-v$ satisfies conditions (a) and (b),
we have $k+\ell = \sum_{i=1}^n |a_i| =  \sum_{i=1}^n |b_i|   =k'+\ell'$.
Hence the weight of $u$ and $v$ are the same with respect to $\wb_{\rm card}$.
Thus the initial monomial of each binomial is the first monomial.
Assume that $\Gc$ is not a Gr\"obner basis of $I_{\Ec_P}$
 with respect to $<_P$.
Let ${\rm in}(\Gc) = \left< {\rm in}_{<_P} (g)  : g \in \Gc \right>$.
Then there exists a non-zero irreducible homogeneous binomial $f = u-v \in I_{\Ec_P}$
such that neither $u$ nor $v$ belongs to ${\rm in}(\Gc)$.
Since both $u$ and $v$ cannot be divided by
$x_{i_1,\dots,i_k}^{(\varepsilon_1,\dots, \varepsilon_k)} x_{j_1,\dots, j_\ell}^{(\mu_1,\dots, \mu_\ell)}$ with $i_p = j_q$ and $\varepsilon_p \neq \mu_q$ for some $p$, $q$, they are of the form
$$
u= 
x_{i_1,\dots,i_{k_1}}^{(\varepsilon_1,\dots, \varepsilon_{k_1})}
x_{i_{k_1+1},\dots,i_{k_2}}^{(\varepsilon_{k_1+1},\dots, \varepsilon_{k_2})}
\dots
x_{i_{k_{r-1}+1},\dots,i_{k_r}}^{(\varepsilon_{i_{k_{r-1}+1}},\dots, \varepsilon_{k_r})},
\ 
v= x_{j_1,\dots,j_{\ell_1}}^{(\mu_1,\dots, \mu_{\ell_1})}
x_{j_{\ell_1+1},\dots,j_{\ell_2}}^{(\mu_{\ell_1+1},\dots, \mu_{\ell_2})}
\dots
x_{j_{\ell_{r-1}+1},\dots,j_{\ell_r}}^{(\mu_{j_{\ell_{r-1}+1}},\dots, \mu_{\ell_r})},
$$
where 
\begin{itemize}
\item[(a)]
For any $p,q$ such that $i_p = i_q$,
we have $\varepsilon_p =\varepsilon_q$;
\item[(b)]
For any $p,q$ such that $j_p = j_q$,
we have $\mu_p =\mu_q$.
\end{itemize}
Moreover since $u$ and $v$ cannot be divided by
the initial monomial of a binomial in (\ref{niban}), 
we may assume that, for some $I_1, \dots , I_r, J_1,\dots,J_r \in \Jc(P)$,
we have
$$
u= 
x_{\max(I_1)}^{(\varepsilon_1,\dots, \varepsilon_{k_1})}
x_{\max(I_2)}^{(\varepsilon_{k_1+1},\dots, \varepsilon_{k_2})}
\dots
x_{\max(I_r)}^{(\varepsilon_{i_{k_{r-1}+1}},\dots, \varepsilon_{k_r})},
\ 
v= x_{\max(J_1)}^{(\mu_1,\dots, \mu_{\ell_1})}
x_{\max(J_2)}^{(\mu_{\ell_1+1},\dots, \mu_{\ell_2})}
\dots
x_{\max(J_r)}^{(\mu_{j_{\ell_{r-1}+1}},\dots, \mu_{\ell_r})},
$$
with $I_1 \subset \dots \subset I_r$ and 
$J_1 \subset \dots \subset J_r$.
Since $u$ and $v$ satisfy conditions (a) and (b)
and since $f$ belongs to $I_{\Ec_P}$,
it then follows that 
$x_{\max(I_r)}^{(\varepsilon_{i_{k_{r-1}+1}},\dots, \varepsilon_{k_r})}
=
x_{\max(J_r)}^{(\mu_{j_{\ell_{r-1}+1}},\dots, \mu_{\ell_r})}$.
This contradicts the assumption that $f$ is irreducible.
\end{proof}

By the correspondence 
\cite[Chapter 8]{sturmfels1996}
between a squarefree quadratic initial ideal of $I_{\Ec_P}$
and a flag regular unimodular triangulation of $\Ec_P$, 
Theorem~\ref{thm:flag} follows from Lemma~\ref{genten_yowaku} and 
Theorem~\ref{thm:grobner}.


\section{$\gamma$-positivity and real-rootedness of 
the $h^*$-polynomial of $\Ec_P$}
\label{sec:enriched}

In this section, we discuss the Ehrhart polynomial and $h^*$-polynomial of $\Ec_P$ of a finite poset $P=[n]$.
In particular, we prove Theorems \ref{ehrhartpolynomial} and \ref{maintheorem}.

Let $(P, \omega)$ be a poset with $n$ elements
and let $\Omega'_m =\{1,-1,2,-2,\dots, m ,-m\}$ for $0< m \in \ZZ$. 
A map $f: P \rightarrow \Omega'_m$ is called an 
{\em enriched $(P,\omega)$-partition} (\cite{StembridgeEnriched})
if, for all $x, y \in P$ with $x <_P y$, $f$ satisfies
\begin{itemize}
\item
$|f(x)| \le |f(y)|$;
\item
$|f(x)| = |f(y)| \ \Rightarrow \ f(x) \le f(y)$;
\item
$f(x) = f(y) >0 \ \Rightarrow \ \omega(x) < \omega(y)$;
\item
$f(x) = f(y) < 0  \ \Rightarrow \ \omega(x) > \omega(y)$.
\end{itemize}
In the present paper, we always assume that $(P, \omega)$ is {\em naturally labeled}.
Then the above condition is equivalent to the following conditions:
\begin{itemize}
\item
$|f(x)| \le |f(y)|$;
\item
$|f(x)| = |f(y)|  \ \Rightarrow \ f(y) > 0$.
\end{itemize}
For each $0 < m \in \ZZ$, let $\Omega'_{P}(m)$ denote
the number of  enriched $(P,\omega)$-partitions $f: P \rightarrow \Omega'_m$.
Then $\Omega'_{P}(m)$ is a polynomial in $m$ called the {\em  enriched order polynomial} of $P$.

On the other hand,
Petersen \cite{Petersen} introduced slightly different notion ``left enriched $(P,\omega)$-partitions'' as follows.
Let $\Omega_m^{(\ell)} =\{0,1,-1,2,-2,\dots, m ,-m\}$ for $0< m \in \ZZ$. 
A map $f: P \rightarrow \Omega_m^{(\ell)}$ is called a {\em left enriched $(P,\omega)$-partition} if, for all $x, y \in P$ with $x <_P y$, $f$ satisfies
the the following conditions:
\begin{itemize}
\item[(i)]
$|f(x)| \le |f(y)|$;
\item[(ii)]
$|f(x)| = |f(y)| \ \Rightarrow \ f(y) \ge 0$.
\end{itemize}

For each $0 < m \in \ZZ$, let $\Omega_{P}^{(\ell)}(m)$ denote
the number of left enriched $(P,\omega)$-partitions $f: P \rightarrow \Omega_m^{(\ell)}$.
Then $\Omega_{P}^{(\ell)}(m)$ is a polynomial in $m$  called the {\em left enriched order polynomial} of $P$.
We can compute the left enriched order polynomial $\Omega_{P}^{(\ell)}(m)$ of $P$ from the enriched order polynomial 
$\Omega'_{P}(m)$ of $P$.
In fact, it follows that
\[
\Omega_{P}^{(\ell)}(m)=\dfrac{1}{2}(\Omega'_{P}(m+1)-\Omega'_{P}(m)).
\]

Now, we prove Theorem \ref{ehrhartpolynomial}.
\begin{proof}[Proof of Theorem \ref{ehrhartpolynomial}]
It is enough to construct a bijection from $m \Ec_P \cap \ZZ^n $ 
to the set $F(m)$ of all left enriched $(P,\omega)$-partitions $f: P \rightarrow \Omega_{m}^{(\ell)}$.
Let $\varphi: F(m) \rightarrow m \Ec_P \cap \ZZ^n $ be a map defined by
$
\varphi (f) = (x_1, \dots, x_{n}) \in \{0, \pm 1, \dots, \pm m\}^n,
$
where 
$$
x_i = 
\left\{
\begin{array}{cc}
f(i) &  \mbox{ if } i \mbox{ is minimal in } P,\\
\\
\min \{ |f(i)|  - | f(j)| : i \mbox{ covers } j \mbox{ in } P\} & \mbox{ if } 
 i \mbox{ is not minimal in } P \mbox{ and } f(i) \ge 0,\\
\\
- \min \{| f(i)|  - | f(j)| : i \mbox{ covers } j \mbox{ in } P\} & \mbox{ otherwise}
\end{array}
\right.
$$
for each $f \in F(m)$.
(This map arising from the map defined in \cite[Theorem 3.2]{twoposetpolytopes}.)

\bigskip

\noindent
{\bf Claim 1.} ($\varphi$ is well-defined.)
By condition (i) for a left enriched $(P, \omega)$-partition $f$,
we have $(|x_1|, \dots, |x_{n}|) \in m \Cc_P$
by a result of Stanley \cite[Theorem 3.2]{twoposetpolytopes}.
Hence, by Lemma~\ref{keylemma}, 
$\varphi (f)$ belongs to $m \Ec_P \cap \ZZ^n $.

\bigskip

Let $\psi:m \Ec_P \cap \ZZ^n \rightarrow   F(m)$ be a map 
defined by 
$\psi (\xb) : P \rightarrow \Omega_{m}^{(\ell)}$, where 
$$\psi (\xb) (i) =
\left\{
\begin{array}{cl}
\max\{|x_{j_1}| + \dots +  |x_{j_k}|  :  j_1 <_P \dots <_P j_k =i \} & \mbox{ if } x_i \ge 0,\\
\\
-
\max\{ |x_{j_1}| + \dots +  |x_{j_k}|  : j_1 <_P \dots <_P j_k =i \} & \mbox{ otherwise}
\end{array}
\right.
$$
for each $\xb = (x_1, \dots, x_{n}) \in m \Ec_P \cap \ZZ^n$.

\bigskip

\noindent
{\bf Claim 2.} ($\psi$ is well-defined.)
If $\xb = (x_1, \dots, x_{n})$ belongs to $m \Ec_P \cap \ZZ^n$, then
$ (|x_1|, \dots, |x_{n}|)$ belongs to $m \Cc_P \cap \ZZ^n$ by Lemma~\ref{keylemma}.
Since $\psi$ is an extension of a map given in \cite[Proof of Theorem 3.2]{twoposetpolytopes}, $\psi (\xb)$ satisfies condition (i) in the
definition of a left enriched partition.
Suppose that $i <_P j$,
$|\psi (\xb) (i)| = |\psi (\xb) (j)|$ and  $\psi (\xb) (j) < 0$.
Then $x_j <0$.
Since $i <_P j$ and $|x_j| >0$, we have $|\psi (\xb) (i)| < |\psi (\xb) (j)|$,
a contradiction.
Thus $\psi (\xb)$ is a left enriched $(P,\omega)$-partition.

\bigskip

Finally, we show that $\varphi$ is a bijection.
It is enough to show that $\psi$ is the inverse of $\varphi$.
Let $\varphi(f) = \xb= (x_1,\dots, x_{n})$ for $f \in F(m)$.
By conditions (i) and (ii) for $f$ together with
 the definition of $\varphi$ and $\psi$, we have
$$
f(i) \ge 0  \Longleftrightarrow x_i \ge 0 \Longleftrightarrow \psi(\xb)(i) \ge 0.
$$
Thus the map $\psi \circ \varphi: F(m) \rightarrow F(m)$ satisfies
$$
\psi \circ  \varphi (f) (i) \ge 0
\Longleftrightarrow
f( i )\ge 0 \ \ \  (1 \le i \le n) 
$$
for any $f \in F(m)$.
As stated in  \cite[Proof of Theorem 3.2]{twoposetpolytopes}, 
we have 
$| \psi \circ  \varphi (f) (i) | = |  f( i ) |$
 ($1 \le i \le n$) for any $f \in F(m)$.
Thus $\psi \circ  \varphi $ is an identity map.
Conversely, let $f= \psi(\xb)$ for $\xb = (x_1, \dots, x_{n}) \in m \Ec_P \cap \ZZ^n$ and let $\varphi(f)=(y_1, \dots, y_{n})$.
By conditions (i) and (ii) for $f$ together with
 the definition of $\varphi$ and $\psi$, we have
$$
x_i \ge 0 \Longleftrightarrow f(i) \ge 0 \Longleftrightarrow y_i \ge 0.
$$
Thus the map $\varphi \circ \psi : m \Ec_P \cap \ZZ^n \rightarrow  m \Ec_P \cap \ZZ^n$ satisfies
$$x_i \ge 0 \Longleftrightarrow y_i \ge 0 \ \ \ (1 \le i \le n).$$
As stated in \cite[Proof of Theorem 3.2]{twoposetpolytopes}, 
we have 
$| x_i | = | y_i |$ for $1 \le i \le n$.
Thus $\varphi \circ  \psi $ is an identity map.
Therefore $\varphi$ is a bijection, as desired.
\end{proof}

Let  $f= \sum_{i=0}^{n}a_i x^i$ be a polynomial with real coefficients and $a_n \neq 0$.
We now focus on the following properties.
\begin{itemize}
	\item[(RR)] We say that $f$ is {\em real-rooted} if all its roots are real.  
	\item[(LC)] We say that $f$ is {\em log-concave} if $a_i^2 \geq a_{i-1}a_{i+1}$ for all $i$.
	\item[(UN)] We say that $f$ is {\em unimodal} if $a_0 \leq a_1 \leq \cdots \leq a_k \geq \cdots \geq a_n$ for some $k$.
\end{itemize}  
If all its coefficients are nonnegative, then these properties satisfy the implications
\[
{\rm(RR)} \Rightarrow {\rm(LC)} \Rightarrow {\rm(UN)}.
\]
On the other hand, the polynomial $f$ is said to be {\em palindromic} if $f(x)=x^nf(x^{-1})$.
It is {\em $\gamma$-positive} if $f$ is palindromic and there are $\gamma_0,\gamma_1,\ldots,\gamma_{\lfloor n/2\rfloor} \geq 0$ such that $f(x)=\sum_{i \geq 0}
\gamma_i \  x^i (1+x)^{n-2i}$.
The polynomial $\sum_{i \geq 0}\gamma_i \ x^i$ is called {\em $\gamma$-polynomial of $f$}. 
We can see that a $\gamma$-positive polynomial is real-rooted if and only if its $\gamma$-polynomial is real-rooted.
If $f$ is a palindromic and real-rooted, then it is $\gamma$-positive.
Moreover, if $f$ is $\gamma$-positive, then it is unimodal.

In the rest of the present section, we discuss the $\gamma$-positivity and the real-rootedness on the $h^*$-polynomial of $\Ec_P$.
It is known \cite{hibi} that the $h^*$-polynomial of a lattice polytope $\Pc$ with the interior lattice point ${\bf 0}$ is palindromic if and only if $\Pc$ is reflexive. Moreover, if a reflexive polytope $\Pc$ has a regular unimodular triangulation, then the $h^*$-polynomial is unimodal (\cite{BR}).
On the other hand, if a reflexive polytope $\Pc$ has a flag regular unimodular triangulation
such that each maximal simplex 
contains the origin as a vertex, then the $h^*$-polynomial coincides with the $h$-polynomial of a flag triangulation of a sphere.
Hence from Theorem~\ref{thm:grobner}, we can show the following.
\begin{Corollary}
	\label{cor:flag}
		Let $P=[n]$ be a poset.
	Then the $h^*$-polynomial of $\Ec_P$ is palindromic, unimodal,
		and coincides with the $h$-polynomial of a flag triangulation of a sphere.
\end{Corollary}


Given a linear extension $\pi = (\pi_1,\dots,\pi_n)$ of a poset $P=[n]$,
a {\em peak} (resp. {\em a left peak}) of $\pi$ is an index $2 \le i \le n-1$ (resp. $1 \le i \le n-1$) such that 
$\pi_{i-1} <\pi_i > \pi_{i+1} $, where we set $\pi_0 =0$.
Let ${\rm pk}(\pi)$ (resp. ${\rm pk}^{(\ell)}(\pi)$) denote the number of peaks (resp. left peaks) of $\pi$.
Then the {\em peak polynomial} $W_{P} (x)$ and the {\em left peak polynomial} $W_{P}^{(\ell)} (x)$ of $P$ are defined by
$$
W_{P} (x) = 
\sum_{\pi \in {\mathcal L} (P)} x^{\ {\rm pk}(\pi)}
$$
and
$$
W_{P}^{(\ell)} (x) = 
 \sum_{\pi \in {\mathcal L} (P)} x^{\ {\rm pk}^{(\ell)}(\pi)}.
$$
Petersen \cite{Petersen} computed the generating function for a left enriched order polynomial:

\begin{Lemma}[{\cite[Theorem~4.6]{Petersen}}]
	\label{lem:petersen}
	Let $P = [n]$ be a naturally labeled poset. 
	Then we have the following generating function for the left enriched order polynomial of $P${\rm :}
	$$
\sum_{m \geq 0}\Omega^{(\ell)}_P(m) x^m	=  \dfrac{(x+1)^n}{(1-x)^{n+1}} \ W_{P}^{(\ell)}  \left(  \frac{4x}{(x+1)^2} \right).
	$$
\end{Lemma}

Therefore, Theorem \ref{maintheorem} follows from Theorem \ref{ehrhartpolynomial} and Lemma \ref{lem:petersen}.

\begin{Remark}
A poset $P$ is said to be {\em narrow} if 
the vertices of $P$ may be partitioned into two chains.
Stembridge \cite[Proposition 1.1]{Stembridge} essentially pointed out that,
if $P$ is narrow, then $W_{P}^{(\ell)} (x)$ coincides with the {\em $P$-Eulerian polynomial} 
$$W(P)( x)= 
\sum_{\pi \in {\mathcal L} (P)} x^{\ {\rm des}(\pi)},$$
where ${\rm des}(\pi)$ is the number of descents of $\pi$.
Thus, for a narrow poset $P$,
$$h^*(\Ec_P, x) =  (x+1)^n \ W(P) \left(\frac{4x}{(x+1)^2} \right).$$
This fact coincides with the result in \cite{OTinterior} for a bipartite permutation graph.
Note that a naturally labeled narrow poset $P$ such that $ W(P) (x)$ is not real-rooted is given in \cite{Stembridge}.
\end{Remark}

Given a poset $P = [n]$, the {\em comparability graph} $G(P)$ of $P$
is the graph on the vertex set $[n]$ 
with $i, j \in [n]$ adjacent if either $i <_P j$  or $j <_P i$.
Then $\{i_1,\dots,i_k\} \subset [n]$ is an antichain of $P$ if and only if $\{i_1,\dots,i_k\}$ is a stable set (independent set) of $G(P)$.
Hence $\Ec_P = \Ec_{P'}$ if $G(P) = G(P')$ for posets $P$ and $P'$.
Thus we have the following immediately.

\begin{Corollary}
\label{cor:comp}
Both the (left) enriched order polynomial 
of $P$ and  the (left) peak polynomial of $P$
depend only on the comparability graph ${\rm G}(P)$ of $P$.

\end{Corollary}

\section{The $\Gamma$-complexes}
\label{sec:complex}
In \cite{NevoPetersen}, Nevo and Petersen conjectured the following.
\begin{Conjecture}[{\cite[Conjecture 1.4]{NevoPetersen}}]
	\label{conj:NP}
	The $\gamma$-polynomial of any flag triangulation of a sphere is the $f$-polynomial of a simplicial complex.
	\end{Conjecture}
Equivalently, the coefficients of the $\gamma$-polynomial satisfy the Kruskal--Katona inequalities. (See \cite[Chapter II.2]{Stanleygreenbook}.)
Clearly,  Conjecture \ref{conj:NP} is stronger than Gal's Conjecture.
Moreover, they gave the following problem.
\begin{Problem}[{\cite[Problem 6.4]{NevoPetersen}}]
	\label{prob:NP}
	The $\gamma$-polynomial of any flag triangulation of a sphere is the $f$-polynomial of a flag simplicial complex.
\end{Problem}
In this section, we solve this problem  for enriched chain polytopes.

Let $P=[n]$ be an antichain.
Then $h^*(\Ec_P, x) $ coincides with the Eulerian polynomial $B_n(x)$
of type B.
See \cite[Proposition 4.15]{Petersen}.
In this case, a flag simplicial complex $\Gamma({\rm Dec}_n)$
whose $f$-polynomial is the $\gamma$-polynomial
of $h^*(\Ec_P, x) $ is given in \cite[Corollary 4.5 (2)]{NevoPetersen}
as follows.
A {\em decorated permutation} $\wb$ is a permutation $w \in {\mathfrak S}_n$ with
bars colored in four colors $|^0$, $|^1$, $|^2$, and $|^3$ following 
the left peak positions.
Let ${\rm Dec}_n$ be the set of all decorated permutations.
Given $w \in {\mathfrak S}_n$, there exist $4^{{\rm pk}^{(\ell)}(w)}$
decorated permutations associated with $w$ in ${\rm Dec}_n$.
For example, $3|^2 24|^1 157|^0 689$ belongs to ${\rm Dec}_9$.
Given 
\begin{equation}
\label{deco}
\wb=w_1 |^{c_1} \dots |^{c_{i-1}} w_i |^{c_i} w_{i+1} |^{c_{i+1}} \dots |^{c_{\ell-1}}w_\ell \in {\rm Dec}_n,
\end{equation}
let $w_i = \grave{w}_i \acute{w}_i$ where $\grave{w}_i$ is the decreasing part of $w_i$
and $\acute{w}_i$ the increasing part of $w_i$.
We say that $\wb \in {\rm Dec}_n$ covers $\ub \in {\rm Dec}_n$
if and only if $\ub$ is obtained from $\wb$ by removing a colored bar $|^{c_i}$ and
reordering the word $w_i w_{i+1} = \grave{w}_i \acute{w}_i w_{i+1}$ as a word $\grave{w}_i \acute{a}$
where $\acute{a} = {\rm sort}( \acute{w}_i w_{i+1} )$.
Then $({\rm Dec}_n, \le)$ is a poset graded by number of bars.

We associate the set ${\rm Dec}_n$ with the flag simplicial complex $\Gamma({\rm Dec}_n)$ on the vertex set $V=\{\wb \in {\rm Dec}_n : {\rm pk}^{(\ell)}(w) = 1\}$.
In $\Gamma({\rm Dec}_n)$, two vertices
$\ub = \acute{u}_1 |^c \grave{u}_2 \acute{u}_2$ and 
$\vb = \acute{v}_1 |^d \grave{v}_2 \acute{v}_2$ with
$|\acute{u}_1| < | \acute{v}_1 |$ are adjacent if and only if
$\wb = \acute{u}_1 |^c \grave{u}_2 \acute{a}  |^d \grave{v}_2 \acute{v}_2$
belongs to ${\rm Dec}_n$,
where $\acute{a} = {\rm sort}(\acute{u}_2 \cap \acute{v}_1 )$.
Then $\Gamma({\rm Dec}_n)$ is the collection of all subsets $F$ of $V$ such that
every two distinct vertices in $F$ are adjacent.
By definition, $\Gamma({\rm Dec}_n)$ is a flag simplicial complex.
Let $\phi: {\rm Dec}_n \rightarrow \Gamma({\rm Dec}_n)$
be a map defined by 
$$\phi(\wb) = \{ 
w_1 |^{c_1} \grave{w}_2 \acute{b}_1 ,
\dots, 
\acute{a}_i |^{c_i} \grave{w}_{i+1} \acute{b}_i,
\dots, 
\acute{a}_{\ell-1} |^{c_{\ell-1}} \grave{w}_{\ell} \acute{b}_{\ell-1}
\}$$
for 
$\wb=w_1 |^{c_1} \dots |^{c_{i-1}} w_i |^{c_i} w_{i+1} |^{c_{i+1}} \dots |^{c_{\ell-1}}w_\ell$,
where 
$\acute{a}_i$ is the set of letters to the left of $\grave{w}_{i+1}$ in $\wb$ written in increasing order
and $\acute{b}_i$ is the set of letters to the right of $\grave{w}_{i+1}$ in $\wb$ written in increasing order.
It was shown \cite{NevoPetersen} that $\phi$ is an isomorphism of graded posets
from $({\rm Dec}_n, \le)$ to $(\Gamma({\rm Dec}_n ) , \subseteq)$.
Thus we have the following (\cite[Corollary 4.5 (2)]{NevoPetersen}).

\begin{Proposition}
Let $P=[n]$ be an antichain.
Then the $\gamma$-polynomial of $h^*(\Ec_P, x) $ is the $f$-polynomial of the 
flag simplicial complex $\Gamma({\rm Dec}_n)$.
\end{Proposition}

Given a poset $P=[n]$, 
let $S_P = \{\wb \in {\rm Dec}_n : w \in {\mathcal L}(P)\}$.
%
Then we have the following.

\begin{Theorem}
	\label{thm:fpoly}
Let $P=[n]$ be a poset.
Then the image $\Gamma(S_P) := \phi(S_P)$ is a flag simplicial subcomplex 
of $\Gamma({\rm Dec}_n)$ whose $f$-polynomial is
the $\gamma$-polynomial of $h^*(\Ec_P, x) $.
\end{Theorem}

\begin{proof}
First, we show that $\Gamma(S_P)$ is a subcomplex of $\Gamma({\rm Dec}_n)$.
Let $\wb$ of the form (\ref{deco}) be an element of $S_P$.
If $\ub$ is obtained from $\wb$ by removing a colored bar $|^{c_i}$ and
reordering the word $w_i w_{i+1} = \grave{w}_i \acute{w}_i w_{i+1}$ as a word $\grave{w}_i \acute{a}$
where $\acute{a} = {\rm sort}( \acute{w}_i w_{i+1} )$, then $u \in {\mathfrak S}_n$ is obtained from $w \in {\mathcal L}(P)$
by sorting a consecutive part of $w$.
Then $u \in {\mathcal L}(P)$, and hence $\ub \in S_P$.
Thus $(S_P, \le)$ is a lower ideal in $({\rm Dec}_n, \le)$.
Since $\phi$ is an isomorphism of graded posets, it follows that
$\Gamma(S_P)$ is a subcomplex of $\Gamma({\rm Dec}_n)$.

Second, we show that $\Gamma(S_P)$ is flag.
Let $V_P =\{\wb \in S_P :{\rm pk}^{(\ell)}(w) = 1\}$.
Since $(S_P, \le)$ is a lower ideal, $\phi(\wb) \subset V_P$
if $\wb \in S_P$.
Let $F=\{\ub_1,\dots, \ub_\ell\}$ be pairwise adjacent vertices in $V_P$
ordered by increasing position of the bar in $\ub_i$.
We show that $\phi^{-1} (F)$ belongs to $S_P$.
If $\ell=1$, then it is trivial.
Suppose by induction on $\ell$ that 
$$
\phi^{-1} (\{\ub_1,\dots, \ub_{\ell-1}\}) =
\wb=w_1 |^{c_1} \dots |^{c_{i-1}} w_i |^{c_i} w_{i+1} |^{c_{i+1}} \dots |^{c_{\ell-1}}\grave{w}_\ell \acute{w}_\ell \in S_P
$$
with $\ub_{\ell-1} = \acute{u}_{\ell-1,1}  |^{c_{\ell-1}}\grave{w}_\ell \acute{w}_\ell$.
Then $\phi^{-1} (F)$ is 
$$
\wb'= w_1 |^{c_1} \dots |^{c_{i-1}} w_i |^{c_i} w_{i+1} |^{c_{i+1}} \dots |^{c_{\ell-1}}\grave{w}_\ell \acute{a} |^{c_\ell} \grave{u}_{\ell,2} \acute{u}_{\ell,2},
$$
where $\ub_\ell = \acute{u}_{\ell,1} |^{c_\ell} \grave{u}_{\ell,2} \acute{u}_{\ell,2}$
and
$\acute{a} = {\rm sort} (\acute{w}_\ell  \cap \acute{u}_{\ell,1})$.
Since both $\wb$ and $\ub_\ell$ belong to $S_P$ and
since $\acute{a} \cup \grave{u}_{\ell,2} \cup \acute{u}_{\ell,2} \subset \acute{w}_\ell$
and $\acute{u}_{\ell-1,1}  \cup \grave{w}_\ell \cup \acute{a}  \subset \acute{u}_{\ell,1}$, it follows that $\wb'$ belongs to $S_P$.

From Theorem~\ref{maintheorem}, 
the $\gamma$-polynomial of $h^*(\Ec_P, x) $ is
$$
W_{P}^{(\ell)}(4x) 
=  
\sum_{\pi \in {\mathcal L} (P)} 4^{\ {\rm pk}^{(\ell)}(\pi)} x^{\ {\rm pk}^{(\ell)}(\pi)}
$$
Since $4^{\ {\rm pk}^{(\ell)}(\pi)}$ is the number of
decorated permutations associated with $\pi$ in $S_P$,
this is equal to the $f$-polynomial of $\Gamma(S_P)$ as desired.
\end{proof}


\begin{thebibliography}{99}

	\bibitem{mirror}
V.~Batyrev, 
Dual polyhedra and mirror symmetry for Calabi-Yau hypersurfaces in toric varieties,  
{\em J. Algebraic Geom.}, {\bf 3} (1994), 493--535. 

\bibitem{represent}
R. Biswal and G. Fourier,
Minuscule Schubert Varieties: Poset Polytopes, PBW-Degenerated Demazure Modules, and Kogan Faces,
{\em Algebr. Represent. Theory}, {\bf 18} (2015), 1481--1503.

\bibitem{BR}
W. Bruns and T. R\"omer,
$h$-Vectors of Gorenstein polytopes,
\textit{J. Combin. Theory Ser. A} {\bf 114} (2007), 65--76.


\bibitem{Cox}
D.~Cox, J.~Little and H.~Schenck, 
``Toric varieties", 
{Amer. Math. Soc.}, 2011. 

\bibitem{Ehrhart}
E. Ehrhart, 
``Polynom\^{e}s Arithm\'{e}tiques et M\'{e}thode des Poly\'{e}dres en Combinatorie", 
Birkh\"{a}user, Boston/Basel/Stuttgart, 1977. 

\bibitem{Gal}
S. R. Gal,
Real Root Conjecture fails for five and higher dimensional spheres,
\textit{Discrete Comput. Geom.}, {\bf 34} (2005), 269--284.

\bibitem{hibi}
T. Hibi, Dual polytopes of rational convex polytopes, 
\textit{Combinatorica} {\bf 12} (1992), 237--240.


\bibitem{HibiLi}
T.~Hibi and N.~Li,
Chain polytopes and algebras with straightening laws,
\textit{Acta Math. Vietnam.} {\bf 40} (2015), 447--452.

\bibitem{HMOS}
T.~Hibi, K.~Matsuda, H.~Ohsugi, and K.~Shibata,
Centrally symmetric configurations of order polytopes,
\textit{J. Algebra} {\bf 443} (2015), 469--478.

\bibitem{HTomega}
T. Hibi and A. Tsuchiya, 
Facets and volume of Gorenstein Fano polytopes,
{\em Math. Nachr.} {\bf 290} (2017), 2619--2628.

\bibitem{HTperfect}
T. Hibi and A. Tsuchiya,
Reflexive polytopes arising from perfect graphs,
\textit{J. Combin. Theory Ser. A} {\bf 157} (2018), 233--246.


\bibitem{Kre}
M. Kreuzer and H. Skarke,
Complete classification of reflexive polyhedra in four dimensions,
\textit{Adv. Theor. Math. Phys.} \textbf{4}(2000), 1209--1230.

\bibitem{Lag}
J. C. Lagarias and G. M. Ziegler,
Bounds for lattice polytopes containing a fixed number of interior points in a sublattice,
\textit{Canad. J. Math.} \textbf{43}(1991), 1022--1035.

\bibitem{NevoPetersen}
E. Nevo and T. K. Petersen,
On $\gamma$-vectors satisfying the Kruskal--Katona inequalities,
\textit{Discrete Comput. Geom.}, {\bf 45} (2011), 503--521.


	\bibitem{harmony}
H.~Ohsugi and T.~Hibi,
Reverse lexicographic squarefree initial ideals and Gorenstein Fano polytopes,
{\em J. Commut. Alg.} {\bf 10} (2018), 171--186.

\bibitem{OTinterior}
H.~Ohsugi and A.~Tsuchiya, Reflexive polytopes arising from bipartite graphs with $\gamma$-positivity associated to interior polynomials.  {\tt arXiv:1810.12258} 


\bibitem{Petersen}
T.~K.~Petersen,
Enriched $P$-partitions and peak algebras,
\textit{Adv. Math.} {\bf 209} (2007) 561--610.

\bibitem{Stanleynonnegative}
R.~P. Stanley,
Decompositions of rational convex polytopes,
\textit{Annals of Discrete Math.} {\bf 6} (1980), 333--342. 

		\bibitem{twoposetpolytopes}
	R.~P.~Stanley,
	Two poset polytopes, 
	{\em Disc. Comput. Geom.} {\bf 1} (1986), 9--23.

\bibitem{Stanleygreenbook}
R.~P. Stanley,
Combinatorics and Commutative Algebra, 2nd edn, Birkh\"{a}user, Boston (1996). 

\bibitem{StembridgeEnriched}
J.~R.~Stembridge, Enriched $P$-partitions, 
\textit{Trans. Amer. Math. Soc.} {\bf 349} (1997), 763--788.

\bibitem{Stembridge}
J.~R.~Stembridge,
Counterexamples to the poset conjectures of Neggers, Stanley, and Stembridge,
\textit{Trans. Amer. Math. Soc.} {\bf 359} (2007), 1115--1128.

	\bibitem{sturmfels1996}
B.~Sturmfels, 
``Gr\"{o}bner bases and convex polytopes," 
Amer. Math. Soc., Providence, RI, 1996. 

	\bibitem{Sul}
S. Sullivant,
Compressed polytopes and statistical disclosure limitation,
{\it Tohoku Math. J.} {\bf 58} (2006), 433--445.


\end{thebibliography}
\end{document}